\documentclass[11pt,reqno,oneside]{amsart}

\usepackage[a4paper, margin=1.5cm]{geometry}

\usepackage{graphicx}
\usepackage{amssymb}
\usepackage{epstopdf}
\usepackage{cases}
\usepackage{hyperref}
\usepackage[abbrev,nobysame]{amsrefs} \AtBeginDocument{\def\MR#1{}} 

\usepackage{amsmath,amsfonts,amsthm,mathrsfs,amssymb,cite}
\usepackage[usenames]{color}

\usepackage{empheq}
\newcommand*\widefbox[1]{\fbox{\hspace{2em}#1\hspace{2em}}}

\newtheorem{lem}{Lemma}[section]
\newtheorem{prop}{Proposition}[section]
\theoremstyle{definition}

\theoremstyle{remark}

\newtheorem{rem}{Remark}[section]
\numberwithin{equation}{section}

\numberwithin{equation}{section}

\usepackage{tikz}
\allowdisplaybreaks

\newcommand{\R}{\mathbb{R}} \newcommand{\mathR}{\mathbb{R}}
\newcommand{\Rn}{{\mathR^n}}

\newcommand{\N}{\mathbb{N}}

\newcommand{\calF}{\mathcal{F}}

\newcommand{\supp}{\mathop{\rm supp}}

\newcommand{\dif}[1]{\,\mathrm{d}{#1}} 
\newcommand{\nrm}[2][]{ \| {#2} \|_{#1}}
\newcommand{\agl}[1][\cdot]{ \langle {#1} \rangle}
\newcommand{\comment}[1]{} 		

\usepackage{tikz}

\newcommand{\sgn}{\mathop{\rm sgn}}

\title{Stationary phase lemmas for general functions}
\author{Shiqi Ma}
\address{Department of Mathematics and Statistics, University of Jyv\"askyl\"a}
\email{mashiqi01@gmail.com, shiqi.s.ma@jyu.fi}
\date{September 13, 2020}

\begin{document}

\begin{abstract}
    We give some details about the stationary phase lemma. We first prove a special case where the high order terms are derived explicitly. Based on that, we prove a more general case by using Morse lemma.
\end{abstract}

\maketitle


\section{Lemma Statements}

\begin{lem} \label{lem:2-staphlema}
	Let $n \in \N^+$ be the dimension.
	We consider the integral $I(\lambda)$:
	\[
	I(\lambda) = \int_{\Rn} e^{i\lambda \agl[Q(x - x_0), x - x_0]/2} a(x;\lambda) \dif{x},
	\]
	where $\agl[Q x, y]$ signifies $(Qx)^T \cdot y$ as matrix multiplication.
	For arbitrary integers $M$, $N \in \mathbb N$, assume
	\begin{itemize}
	
	\item for each $\lambda$, $a(\cdot;\lambda) \in C^{n+2N+3}(\Rn;\mathbb C)$;
	
	\item for each $\lambda$, $a(\cdot;\lambda) \in C^{M}(\Rn;\mathbb C)$, and $\forall \alpha : |\alpha| \leq M$, there exists $\lambda$-dependent constants $C_{M,\alpha}(\lambda) > 0$ such that $\forall x \in \Rn$ there holds $|\partial_x^\alpha a(x - x_0; \lambda)| < C_{M,\alpha}(\lambda) \agl[x]^{2M-n-1-|\alpha|}$;
	
	\item $Q$ is a non-degenerate, symmetric, real-valued matrix.
	
	\end{itemize}
	Then the integral $I(\lambda)$ is well-defined in the oscillatory integral sense, and as $\lambda \to +\infty$ we have
	\begin{subequations} \label{eq:IQ-staphlema}
	\begin{empheq}[box=\widefbox]{align}
	I(\lambda)
	& = \left( \frac{2\pi}{\lambda} \right)^{n/2} \frac{e^{i\frac{\pi}{4}\sgn Q}} {|\det Q|^{1/2}} \sum_{0 \leq j \leq N} \frac {\lambda^{-j}} {j!} \left( \frac {\agl[Q^{-1}D, D]} {2i} \right)^j a(x_0; \lambda) \nonumber \\
	& \qquad + \mathcal O \big( \lambda^{-\frac{n}{2}-N-1} \times \sum_{|\alpha| \leq n+2N+3} \sup_{B(x_0,1)} |\partial^\alpha a(\cdot; \lambda)| \big) \tag{\ref{eq:IQ-staphlema}} \\
	& \qquad + \mathcal O \big( \lambda^{-M} \times \sum_{|\alpha| \leq M} \sup_{\Rn} \frac {|\partial^\alpha a(x - x_0; \lambda)|} {\agl[x]^{2M-n-1-|\alpha|}} \big), \nonumber
	\end{empheq}
	\end{subequations}
	where $B(x_0,1)$ stands for the ball centered at $x_0$ with radius 1, and $\sgn Q$ stands for the difference between the number of positive eigenvalues and the number of negative eigenvalues of the matrix $Q$. 
	Especially, when $M$ is chosen to be within the interval $[\max\{ \lceil n/2 \rceil + N + 1, n+1\} , n + 2N + 3]$, \eqref{eq:IQ-staphlema} can be simplified to
	\begin{equation} \label{eq:IQs-staphlema}
	\boxed{I(\lambda)
	= \left( \frac{2\pi}{\lambda} \right)^{n/2} \frac{e^{i\frac{\pi}{4}\sgn Q}} {|\det Q|^{1/2}} \sum_{0 \leq j \leq N} \frac {\lambda^{-j}} {j!} \left( \frac {\agl[Q^{-1}D, D]} {2i} \right)^j a(x_0; \lambda) + \mathcal O \big( \lambda^{-\frac{n}{2}-N-1} \sum_{|\alpha| \leq n+2N+3} \sup_{\Rn} |\partial^\alpha a(\cdot; \lambda)| \big).}
	\end{equation}
\end{lem}

\begin{rem} \label{rem:1-staphlema}
	The first thing to note is that, in contrast to many other versions of the stationary phase lemma, here we don't require $a$ to be compactly supported.
	Instead, some boundedness conditions are required, which generates the compact-support condition.	
	Second, the integers $M$ and $N$ in Lemma \ref{lem:2-staphlema} shall be chosen properly to serve for your own purposes.
	For example, if one cares more about the decaying behavior w.r.t.~$\lambda$, then the $M$ can be set to $\lceil n/2 \rceil +N+1$.
	However, if one is dealing with these functions $a$ which doesn't have good decaying behavior at the infinity, then one could set $M$ to be large enough such that $\agl[x]^{2M-n-1-|\alpha|}$ can dominate $\partial^\alpha a$, with the cost that we should demand more smoothness of $a(x)$.
	Third, the unit ball $B(x_0,1)$ involved in the term $\sup_{B(0,1)} |\partial^\alpha a|$ can be changed to other bounded domain containing the origin.
	But one should be careful that when the domain you want to use has very small radius, the underlying coefficients of the $\mathcal O(\cdot)$ notation will be very large accordingly.
	Fourth, the function $a$ can depend on $\lambda$, hence the expression \eqref{eq:IQ-staphlema} is an asymptotics only when $\partial_x^\alpha a(x;\lambda)$ don't change significantly when $\lambda \to +\infty$.
\end{rem}

If chosen $M = n + 2N + 3$, Lemma \ref{lem:2-staphlema} can be simplified as follows.

\begin{prop} \label{prop:2-staphlema}
	Assume $I(\lambda)$ and $Q$ are as stated in Lemma \ref{lem:2-staphlema}. If $a \in C^{n+2N+3}(\Rn;\mathbb C)$, and for $\forall \alpha : |\alpha| \leq n+2N+3$ there holds $\sup_\Rn |\partial^\alpha a(x)| < +\infty$, then \eqref{eq:IQs-staphlema} holds.
\end{prop}
Proposition \ref{prop:2-staphlema} can be extend to a more general case where the phase function is not quadratic.

\begin{lem}[Stationary phase lemma \cite{eskin2011lectures}] \label{lem:1-staphlema}
    We consider the integral $I(\lambda)$:
    \[
    I(\lambda) = \int_{\Rn} a(x) e^{i\lambda \varphi(x)} \dif{x}.
    \]
    For an arbitrary integer $N \in \mathbb N$, assume
    \begin{itemize}
   	\item $a \in C^{n+2N+3}(\Rn;\mathbb C)$ with $\sum_{|\alpha| \leq n + 2N+3} \sup_{\Rn} |\partial^\alpha a| < +\infty$;
   	
   	\item $\varphi \in C^{n+2N+6}(\Rn;\R)$ with $\sum_{|\alpha| \leq n+2N+6} \sup_{\Rn} |\partial^\alpha \varphi| < +\infty$;
   	
   	\item $x_0$ is the only critical point of $\varphi(x)$ on $\supp a(x)$, i.e., $\varphi(x_0) = \nabla \varphi(x_0) = 0$, $\varphi_x(x) \neq 0$ for $x \neq x_0$;
   	
   	\item the Hessian $\varphi_{xx}(x_0) := [\frac{\partial^2\varphi}{\partial x_j \partial x_k}(x_0)]_{j,k=1}^n$ satisfies $\det \varphi_{xx}(x_0) \neq 0$.
    \end{itemize}
Then the integral $I(\lambda)$ is well-defined in the oscillatory integral sense, and as $\lambda \to +\infty$ we have
    \begin{subequations} \label{eq:I-staphlema}
    \begin{empheq}[box=\widefbox]{align}
    I(\lambda)
    & = \left( \frac{2\pi}{\lambda} \right)^{n/2} \frac{e^{i\lambda\varphi(x_0) + i\frac{\pi}{4}\sgn \varphi_{xx}(x_0)}}{|\det \varphi_{xx}(x_0)|^{1/2}} \Big( a(x_0) + \sum_{j=1}^{N} a_j(x_0)\lambda^{-j} \Big) \nonumber \\
    & \quad + \mathcal O \big( \lambda^{-\frac{n}{2}-N-1} \times \sum_{|\alpha| \leq n+2N+3} \sup_{\Rn} |\partial^\alpha a| \times \sum_{|\alpha| \leq n+2N+6} \sup_{\Rn} |\partial^\alpha \varphi| \big), \tag{\ref{eq:I-staphlema}} 
    \end{empheq}
    \end{subequations}
    for some functions $a_j, ~1 \leq j \leq N$.
\end{lem}

\begin{rem}
	Proposition \ref{prop:2-staphlema} is a special case of Lemma \ref{lem:1-staphlema} where $\varphi(x) = \agl[Q(x - x_0), x - x_0]/2$, which guarantees $\varphi_{xx}(x_0) = Q$.
	Lemma \ref{lem:1-staphlema} is not a generalization of Lemma \ref{lem:2-staphlema} because unlike the quadratic phase function in Lemma \ref{lem:2-staphlema}, the phase function $\varphi$ in Lemma \ref{lem:1-staphlema} is not assumed to possess the property that $|\nabla \varphi(x)| \simeq |x|$ as $|x|$ is large.
	However if $\varphi$ is a homeomorphism of $\Rn$, it is possible to generalize Lemma \ref{lem:2-staphlema}.
\end{rem}

We can compute these $a_j(x_0)$ in a more precise way (cf.~\cite[(3.4.11)]{zw2012semi}), and the details are given in \cite[Second proof of Theorem 3.11]{zw2012semi}.
In \cite[Lemma 19.3]{eskin2011lectures} there is also another routine to prove the stationary phase lemma.
\cite[Chapter 5]{dim1999spe} by Dimassi and Sj\"ostrand is also a good reference.

\section{Preliminaries}

We need the Taylor's expansion.
Suppose $f \in C^{N+1}(\Rn; \mathbb C)$, then we have that
\begin{equation} \label{eq:TaylorInt-staphlema}
f(x) = \sum_{|\delta| \leq N} \frac{1}{\delta!} \big( \partial^{\delta}f \big)(x_0) \cdot (x - x_0)^{\delta} + (N+1) \sum_{|\delta| = N+1} \frac{(x - x_0)^{\delta}}{\delta!} \int_0^1 (1-t)^N \big( \partial^{\delta}f \big)(x_0 + t(x-x_0)) \dif{t}.
\end{equation}
We omit the proof.

\smallskip

For a measurable function $u$, there exists a constant $C$ such that $\hat u$ exists and
\begin{equation} \label{eq:Fuu-staphlema}
	\nrm[L^1(\Rn)]{\hat u} \leq C \sum_{|\alpha| \leq n + 1} \nrm[L^1(\Rn)]{\partial^\alpha u},
\end{equation}
as long as the right-hand-side of \eqref{eq:Fuu-staphlema} is defined.

\begin{proof}
	We have
	\begin{align*}
		\int |\hat u(\xi)| \dif \xi
		& \leq \int \agl[\xi]^{-n-1} |\agl[\xi]^{n+1} \hat u(\xi)| \dif \xi
		\leq C \sup_{\Rn} |\agl[\xi]^{n+1} \hat u(\xi)| \\
		& \leq C \sum_{|\alpha| \leq n + 1} C_\alpha \sup_{\Rn} |\xi^\alpha \hat u(\xi)|
		\leq \sum_{|\alpha| \leq n + 1} C_\alpha \sup_{\Rn} |(\partial^\alpha u)^\wedge(\xi)| \\
		& \leq \sum_{|\alpha| \leq n + 1} C_\alpha \nrm[L^1(\Rn)]{\partial^\alpha u}.
	\end{align*}
	We arrive at the conclusion.
\end{proof}

\smallskip

We also need the following transformation.
For a fixed non-degenerate, symmetric, real-valued square matrix $Q$, we have
\begin{equation} \label{eq:FrGauiQ-staphlema}
\calF \{ e^{\pm i\,\agl[Q \cdot, \cdot]/2} \}(\xi) = \frac {e^{\pm i \frac \pi 4 {\rm sgn}\, Q}} {|\det Q|^{1/2}} e^{\mp i \agl[Q^{-1} \xi,\xi]/2}.
\end{equation}

\section{Proof of the lemma} \label{sec:pof-staphlema}

We first prove the quadratic case.

\begin{proof}[Proof of Lemma \ref{lem:2-staphlema}]
	
	We omit notationally the dependence of $a$ on $\lambda$ until related clarifications are needed.
	Without loss of generality we assume $x_0 = 0$, and $a(0) = 1$.
	For readers' convenient we rewrite the expression of $I$ here: $I(\lambda) = \int_{\Rn} e^{i\lambda \agl[Qx, x]/2} a(x) \dif{x}$.
	

	According to the assumption on $Q$, we know there exists a decomposition $Q = P \Lambda P^T$ where $P$ is an orthogonal matrix and $\Lambda := (\alpha_j)_{j=1,\cdots,n}$ is a diagonal matrix.
	Make the change of variable $y = P^T x$, we can have
	\begin{align}
	I
	& = \int_\Rn a(Py) e^{i\lambda \sum_{j=1}^{n} \alpha_j y_j^2/2} \dif y \nonumber \\
	& = \int_\Rn e^{i\lambda \sum_{j=1}^{n} \alpha_j y_j^2/2} (1-\chi(y)) f(y) \dif y + \int_\Rn e^{i\lambda \sum_{j=1}^{n} \alpha_j y_j^2/2} \chi(y) f(y) \dif y \qquad (\det P = 1) \nonumber \\
	& =: J_1 + J_2, \label{eq:IQJ12-staphlema}
	\end{align}
	where $f(y) := a(Py)$ and $\chi \in C_c^\infty(\Rn)$ is a cutoff function satisfying $\chi \equiv 1$ in a neighborhood of the origin.
	
	Noticing that a neighborhood of the origin is not included in the support of the integrand in $J_1$, we can estimate $J_1$ by using integration by parts (in the oscillatory integral sense).
	For any integer $M \in \N$ we have
	\begin{align}
	|J_1|
	& = |\int_\Rn \big( \frac {\sum_j \alpha_j^{-1} y_j \partial_j} {i\lambda |y|^2} \big)^M (e^{i\lambda \sum_{j=1}^{n}\alpha_j y_j^2/2} ) \big[ (1-\chi(y)) f(y) \big] \dif y| \nonumber \\
	& \lesssim \lambda^{-M} \int_\Rn |\big( \sum_j \partial_j \circ (y_j |y|^{-2}) \big)^M \big[ (1-\chi(y)) f(y) \big] | \dif y \nonumber \\
	& \lesssim \lambda^{-M} \int_{\supp (1-\chi)} \sum_{|\alpha| \leq M} C_{M; \alpha} |y|^{|\alpha|-2M} |\partial^\alpha ((1-\chi) f(y))| \dif y \label{eq:J1Ind-staphlema} \\
	& \lesssim \lambda^{-M} \big( C \sum_{|\alpha| \leq M} \sup_{\{0 < \chi < 1\}} |\partial^\alpha f| + \sum_{|\alpha| \leq M} C_{M; \alpha} \int_{\{\chi = 0\}} |y|^{|\alpha|-2M} |\partial^\alpha f(y)| \dif y \big) \label{eq:J1Ine-staphlema} \\
	& \lesssim \lambda^{-M} \sum_{|\alpha| \leq M} \big( \sup_{\{0 < \chi < 1\}} |\partial^\alpha a| + \int_{\{\chi = 0\}} |y|^{-n-1} \agl[y]^{|\alpha|-2M+n+1} |\partial^\alpha a(y)| \dif y \big) \nonumber \\
	& \lesssim \lambda^{-M} \sum_{|\alpha| \leq M} \big( \sup_{\{0 < \chi < 1\}} |\partial^\alpha a| + \sup_{\Rn} \frac {|\partial^\alpha a(y)|} {\agl[y]^{2M-n-1-|\alpha|} } \big) \nonumber \\
	& \lesssim \lambda^{-M} \sum_{|\alpha| \leq M} \sup_{\Rn} \frac {|\partial^\alpha a(y)|} {\agl[y]^{2M-n-1-|\alpha|}}. \label{eq:J1f-staphlema}
	\end{align}
	The inequality \eqref{eq:J1Ind-staphlema} is due to the fact that
	\[
	\big( \sum_j \partial_j \circ (y_j |y|^{-2}) \big)^M \varphi = \sum_{|\alpha| \leq M} C_{M; \alpha} |y|^{|\alpha|-2M} \partial^\alpha \varphi
	\]
	which can be derived by induction and we omit the details.
	Inequality \eqref{eq:J1Ine-staphlema} is due to the fact that $\partial^\alpha ((1-\chi) f) = \partial^\alpha f$ in $\{\chi = 0\}$.

	Then we turn to $J_2$.
	Keep in mind that $f(y) = a(Py)$ and $\chi f$ is compactly support and in $C_c^{n+2N+3}(\Rn)$.
	Here we analyze $I$ by borrowing idea from \cite[First proof of Theorem 3.11]{zw2012semi}.
	First we use Plancherel theorem (which states $(f,g) = (\hat f, \hat g)$),
	\begin{align}
	J_2
	& = \int_\Rn \overline{ e^{-i\lambda \sum_{j=1}^{n} \alpha_j y_j^2/2} } \chi f(y) \dif y
	= \int_\Rn \overline{ \calF \{ e^{-i\lambda \sum_{j=1}^{n} \alpha_j (\cdot)^2/2} \}(\xi) } \cdot \widehat{\chi f}(\xi) \dif \xi \nonumber \\
	& = \int_\Rn \overline{ (\lambda)^{-n/2} \frac {e^{-\frac {i\pi} 4 {\rm sgn}\, Q}} {|\det Q|^{1/2}} e^{\frac i {2\lambda} \alpha_j^{-1} \xi_j^2} } \cdot\widehat{\chi f}(\xi) \dif \xi \qquad (\text{by~} \eqref{eq:FrGauiQ-staphlema})\nonumber \\
	& =: \big( \frac {2\pi} \lambda \big)^{n/2} \frac {e^{\frac {i\pi} 4 {\rm sgn}\, Q}} {|\det Q|^{1/2}} J(1/\lambda, 1/\lambda, \chi f), \label{eq:I1J-staphlema}
	\end{align}
	where we ignored the summation notation over $j$ and the function $J$ is defined by
	\begin{equation} \label{eq:Jh-staphlema}
	J(h_1, h_2, \chi f) := (2\pi)^{-n/2} \int_\Rn e^{\frac {\xi_j^2 h_1} {i2\alpha_j}} \cdot \widehat{\chi f}(\xi; 1/h_2) \dif \xi.
	\end{equation}
	Note that in \eqref{eq:Jh-staphlema} we put emphasize on the dependence of $f$ on $h_2$ (i.e.~dependence of $a$ on $\lambda$).
	The smoothness of $J$ w.r.t.~$h_1$ is guaranteed by the $L^1$ of derivatives of $f$, namely, we have the following claim whose justification will be clear in \eqref{eq:JhR-staphlema},
	\[
	\forall m \in \mathbb N,\, \max_{|\alpha| \leq n + 2m + 1} \nrm[L^1(\Rn)]{\partial^\alpha f} < +\infty \ \Rightarrow \ J(\cdot, h_2, f) \in C^m(\R).
	\]
	
	We abbreviate $\partial_{h_1} J$ as $\partial_1 J$.
	Note that
	\begin{align*}
	\partial_1^k J(0, h_2, \chi f)
	& = (2\pi)^{-n/2} \int_\Rn \partial_1^{k} (e^{\frac {\xi_j^2 h_1} {i2\alpha_j}}) |_{h_1 = 0} \cdot \widehat{\chi f}(\xi) \dif \xi
	= (2\pi)^{-n/2} \int_\Rn (\sum_j \frac {\xi_j^2} {i2\alpha_j})^k \cdot \widehat{\chi f}(\xi) \dif \xi \\
	& = (2\pi)^{-n/2} \int_\Rn \widehat{T^k \chi f}(\xi) \dif \xi
	= T^k f(0)
	= T^k \big( a(Py) \big) |_{y = 0} \\
	& = (\frac i 2 P^{lj} P^{kj} \alpha_j^{-1} \partial_{kl})^k a(0)
	= A^k a(0; 1/h_2),
	\end{align*}
	where $T = \frac i 2 \sum_j \frac {\partial_j^2} {\alpha_j}$ and $A = \frac i 2 (Q^{-1})^{jl} \partial_{jl} = \frac 1 {2i} \agl[Q^{-1}D, D]$ (recall that $D = \frac 1 i \nabla$ is vertical).
	We expand $J$ using Taylor series (i.e.~\eqref{eq:TaylorInt-staphlema}),
	\begin{align}
	J(h, h_2, \chi f)
	& = \sum_{k\leq N} \frac{h^k}{k!} \partial_h^{k}J(0, h_2, \chi f) + \frac{h^{N+1}} {N!} \int_0^1 (1-t)^{N} \cdot \partial_1^{N+1}J(th, h_2, \chi f) \dif{t} \nonumber \\
	& = \sum_{0 \leq k \leq N} \frac{(hA)^k}{k!} a(0; h_2) + \frac{h^{N+1}} {N!} \int_0^1 (1-t)^{N} \cdot \partial_1^{N+1}J(th, h_2, \chi f) \dif{t}. \label{eq:JhExp-staphlema}
	\end{align}
	By invoking \eqref{eq:Jh-staphlema}, the remainder term in \eqref{eq:JhExp-staphlema} can be estimated as
	\begin{align*}
	& \ |\frac{h^{N+1}}{N!} \int_0^1 (1-t)^{N} \cdot \partial_1^{N+1}J(th, h_2, \chi f) \dif{t}| \\
	\leq & \ C_N h^{N+1} \int_\Rn |(\frac {-i} {4\alpha_j} \xi_j^2)^{N+1} \cdot \widehat{\chi f}(\xi; 1/h_2)| \dif \xi
	\leq C_N h^{N+1} \sum_{|\beta| \leq 2N+2} C_\beta \nrm[L^1(\Rn)]{(\partial_x^\beta (\chi f(\cdot; 1/h_2)))^\wedge}.
	\end{align*}
	By using \eqref{eq:Jh-staphlema} and \eqref{eq:Fuu-staphlema}, we can continue
	\begin{align}
	& \ |\frac{h^{N+1}}{N!} \int_0^1 (1-t)^{N} \cdot \partial_1^{N+1}J(th, h_2, \chi f) \dif{t}| \nonumber \\
	\leq & \ C_N h^{N+1} \sum_{\substack{|\alpha| \leq n+1 \\ |\beta| \leq 2N + 2}} \nrm[L^1(\Rn)]{\partial_x^{\alpha + \beta} (\chi f(\cdot; 1/h_2))}
	\leq C_N h^{N+1} \sum_{|\alpha| \leq n+ 2N + 3} \nrm[L^1(\Rn)]{\partial_x^{\alpha} (\chi f(\cdot; 1/h_2))} \nonumber \\
	\leq & \ C_N h^{N+1} \sum_{|\alpha| \leq n+2N+3} \sup_{\supp \chi} |\partial^\alpha a(\cdot; 1/h_2)|. \label{eq:JhR-staphlema}
	\end{align}
	Letting $h = h_2 = 1/\lambda$ and combining \eqref{eq:I1J-staphlema}, \eqref{eq:Jh-staphlema}, \eqref{eq:JhExp-staphlema} and \eqref{eq:JhR-staphlema}, we obtain
	\begin{align}
	J_2
	& =  \left( \frac {2\pi} \lambda \right)^{n/2} \frac {e^{\frac {i\pi} 4 {\rm sgn}\, Q}} {|\det Q|^{1/2}} \sum_{j \leq N} \frac {\lambda^{-j}} {j!} \left( \frac {\agl[Q^{-1}D, D]} {2i} \right)^j a(0; \lambda) + C_N \lambda^{-\frac n 2-N-1} \sum_{|\alpha| \leq n+2N+3} \sup_{\supp \chi} |\partial^\alpha a(\cdot; \lambda)|. \label{eq:J2t-staphlema}
	\end{align}

	Combining \eqref{eq:J2t-staphlema} with \eqref{eq:IQJ12-staphlema}, \eqref{eq:J1f-staphlema}, we have
	\begin{align*}
	I
	& = \left( \frac {2\pi} \lambda \right)^{n/2} \frac {e^{\frac {i\pi} 4 {\rm sgn}\, Q}} {|\det Q|^{1/2}} \sum_{j \leq N} \frac {\lambda^{-j}} {j!} \left( \frac {\agl[Q^{-1}D, D]} {2i} \right)^j a(0; \lambda) \\
	& \qquad + \mathcal O \big( \lambda^{-\frac n 2-N-1} \sum_{|\alpha| \leq n+2N+3} \sup_{\supp \chi} |\partial^\alpha a(\cdot; \lambda)| \big) + \mathcal O \big( \lambda^{-M} \sum_{|\alpha| \leq M} \sup_{\Rn} \frac {|\partial^\alpha a(y; \lambda)|} {\agl[y]^{2M-n-1-|\alpha|}} \big),
	\end{align*}
	which is \eqref{eq:IQ-staphlema}.
	
	By setting $M$ to be $\max\{ \lceil n/2 \rceil + N + 1, n+1\} \leq M \leq n + 2N + 3$, we can have $-M \leq -n/2 - N - 1$, $M \leq n + 2N + 3$ and $2M - n - 1 - |\alpha| \geq 0$ provided $|\alpha| \leq M$, so we can obtain \eqref{eq:IQs-staphlema} from \eqref{eq:IQ-staphlema}.
	
	The proof is complete.
\end{proof}

Now we prove the more general case.

\begin{proof}[Proof of Lemma \ref{lem:1-staphlema}]
	Without loss of generality we assume $x_0 = 0$, $\varphi(0) = 0$ and $a(0) = 1$.
	Hence by Taylor's expansion \eqref{eq:TaylorInt-staphlema} we have 
	\[
	\varphi(x) 
	= \sum_{j,k \leq n} x_j x_k \int_0^1 (1-t)\, \partial_{jk} \varphi(tx) \dif t
	= x^T \cdot\int_0^1 (1-t) \varphi_{xx}(tx) \dif t \cdot x.
	\]
	Note that $|\varphi_{xx}(0)| \neq 0$ and $|\varphi_{xx}(x)|$ is continuous on $x$ (\underline{$\varphi \in C^2$}), thus there exists a positive constant \boxed{r} such that $|\varphi_{xx}(x)| > |\varphi_{xx}(0)|/2 > 0$ for all $x \in B(0,r)$.
	Fix a cutoff function $\chi \in C_c^\infty(\Rn)$ such that $\supp \chi \subset B(0,r)$ and $\chi \equiv 1$ in $B(0,r/2)$.
	Hence on $B(0,r)$ the matrix $\varphi_{xx}$ is non-degenerate, and on $\supp a \backslash B(0,r)$ the norm of gradient is uniformly bounded from below, i.e.~there exists a constant $C > 0$ such that $|\nabla \varphi(x)| \geq C$ for $\forall x \in \supp a \backslash B(0,r)$.
	We divide $I$ into two parts
	\begin{equation} \label{eq:ISplit-staphlema}
	I(\lambda)
	= \int_{\Rn} (1 - \chi(x)) a(x) e^{i\lambda \varphi(x)} \dif x + \int_{\Rn} \chi(x) a(x) e^{i\lambda \varphi(x)} \dif x
	:= I_1 + I_2,
	\end{equation}
	and we will show that $I_1$ is rapidly decreasing w.r.t.~$\lambda$ while $I_2$ can be analyzed by using Lemma \eqref{lem:2-staphlema}.

	For $I_1$, denote
	\(
	L = \sum_{j=1}^n \frac {\varphi_{x_j}} {|\nabla \varphi|^2} \partial x_j,
	\)
	where $\varphi_{x_j}$ is short for $\partial_{x_j} \varphi$. 
	Then $\frac 1 {i\lambda} L e^{i\lambda \varphi} = e^{i\lambda\varphi} \text{ and } {}^t L f = \sum_{j=1}^n \partial_{x_j} \big( \frac {\varphi_{x_j} f} {|\nabla \varphi|^2} \big)$.
	For any $N \in \mathbb{N}^+$, $I_2$ can be easily estimated as follows (which requires 
	$a \in C^{n+N+1}(\Rn)$ and $\varphi \in C^{n+N+2}(\Rn)$)
	\begin{align}
	I_1
	& = \int_\Rn (1-\chi) a \cdot ((i\lambda)^{-n-N-1} L^{N+1} e^{i\lambda \varphi(x)}) \dif{x}
	= (i\lambda)^{-n-N-1} \int_\Rn ({}^t L)^{n+N+1} ((1-\chi)a) \cdot e^{i\lambda \varphi(x)} \dif{x} \nonumber\\
	& = \mathcal{O} (\lambda^{-n-N-1} \sum_{|\alpha| \leq n+N+1} \nrm[L^1(\Rn)]{\partial^\alpha a}), \quad \lambda \to \infty \label{eq:I2-staphlema}.
	\end{align}
	As mentioned before, due to the presence of $1-\chi$, the denominator $|\nabla \varphi|^2$ in $L$ keeps a positive distance away from 0, guaranteeing that $({}^t L)^N ((1-\chi)a)$ is bounded and compactly supported.

	Now we turn to $I_2$.
	Because $\varphi \in C^2(\Rn)$, $\varphi_{xx}(x)$ is symmetric and thus there exist orthogonal matrix $P(x)$ and diagonal matrix $\Lambda(x) = (\alpha_j(x))_{j=1,\cdots,n}$ such that 
	\begin{equation*} 
	2 \int_0^1 (1-t) \varphi_{xx}(tx) \dif t = P(x) \Lambda(x) P^T(x).
	\end{equation*}
	Especially we have
	\( 
	P(0) \Lambda(0) P^T(0) = \varphi_{xx}(0).
	\)
	Denote $\alpha_j = \alpha_j(0)$ and $n \times n$ matrix $\Lambda := (\alpha_j)_{j=1,\cdots,n}$ for short. Thus
	\begin{equation*} 
	\Lambda(x) = (\sqrt{\frac {\alpha_j(x)} {\alpha_j}})_{j=1,\cdots,n} \cdot \Lambda \cdot (\sqrt{\frac {\alpha_j(x)} {\alpha_j}})_{j=1,\cdots,n}.
	\end{equation*}
	Make the change of variable:
	\begin{equation} \label{eq:xToy-staphlema}
	y = \Phi(x) := \left( \sqrt{\frac {\alpha_j(x)} {\alpha_j}} \right)_{j=1,\cdots,n} \cdot P^T(x) \cdot x.
	\end{equation}
	Note that
	\begin{equation} \label{eq:sm1-staphlema}
	\varphi \in C^{n+2N+6} \Rightarrow \Phi \in C^{n+2N+4}.
	\end{equation}
	We have
	\begin{align*}
	\varphi(x)
	& = \frac 1 2 x^T \cdot \big[ 2 \int_0^1 (1-t) \varphi_{xx}(tx) \dif t \big] \cdot x
	= \frac 1 2 x^T \cdot \big[ P(x) \Lambda(x) P^T(x) \big] \cdot x \\
	& = \frac 1 2 [P^T(x) \cdot x]^T \cdot (\sqrt{\frac {\alpha_j(x)} {\alpha_j}})_{j=1,\cdots,n} \cdot \Lambda \cdot (\sqrt{\frac {\alpha_j(x)} {\alpha_j}})_{j=1,\cdots,n} \cdot [P^T(x) \cdot x] \\
	& = \frac 1 2 \big[ \big( \sqrt{\frac {\alpha_j(x)} {\alpha_j}} \big)_{j=1,\cdots,n} \cdot P^T(x) \cdot x \big]^T \cdot \Lambda \cdot \big[ \big( \sqrt{\frac {\alpha_j(x)} {\alpha_j}} \big)_{j=1,\cdots,n} \cdot P^T(x) \cdot x \big] \\
	& = \frac 1 2 \agl[\Lambda y, y].
	\end{align*}
	
	We have $\Phi(0) = 0$. It is easy to check that $\frac {\partial \Phi} {\partial x} (0) = P^T(0).$
	From \eqref{eq:xToy-staphlema} it is clear that there exists inverse of $\Phi$, i.e.~$\phi = \Phi^{-1}$
	Note that $x = \phi(\Phi(x))$ and
	\begin{equation} \label{eq:sm2-staphlema}
	\Phi \in C^{n+2N+4} \Rightarrow \phi \in C^{n+2N+4}.
	\end{equation}
	We have
	\begin{align*}
	I_1
	& = \int_\Rn \chi(\phi(y)) a(\phi(y)) \cdot e^{i\lambda \agl[\Lambda y, y]/2} \dif{\phi(y)} \nonumber\\
	& = \int_\Rn \chi(\phi(y)) a(\phi(y)) |\det \nabla_y \phi (y)| \cdot e^{i\lambda \agl[\Lambda y, y]/2} \dif{y} \nonumber\\
	& = \int_\Rn f(y) e^{i\lambda \agl[\Lambda y, y]/2} \dif{y} 
	\end{align*}
	where
	\begin{equation*} 
	f(y) = \chi(\phi(y)) \cdot a(\phi(y)) \cdot |\det \nabla_y \phi (y)|,
	\end{equation*}
	Note that
	\begin{equation} \label{eq:sm3-staphlema}
	\phi \in C^{n+2N+4},\, a \in C^{n+2N+2} \Rightarrow f \in C^{n+2N+2}.
	\end{equation}
	Now we can conclude from \eqref{eq:sm3-staphlema}, \eqref{eq:sm3-staphlema} and \eqref{eq:sm3-staphlema} that
	\begin{equation} \label{eq:sm4-staphlema}
	\varphi \in C^{n+2N+6},\, a \in C^{n+2N+3} \Rightarrow f \in C^{n+2N+3}.
	\end{equation}
	By using Lemma \ref{lem:2-staphlema}, we can obtain
	\begin{align}
	I_1(\lambda)
	& = \left( \frac{2\pi}{\lambda} \right)^{n/2} \frac{e^{i\frac{\pi}{4}\sgn \Lambda}} {|\det \Lambda|^{1/2}} \sum_{0 \leq j \leq N} \frac {\lambda^{-j}} {j!} \left( \frac {\agl[\Lambda^{-1}D, D]} {2i} \right)^j f(0) \nonumber \\
	& \qquad + \mathcal O(\lambda^{-\frac{n}{2}-N-1} \times \sum_{|\alpha| \leq n+2N+3} \sup_{\Rn} |\partial^\alpha f|) \nonumber \\
	& = \left( \frac{2\pi}{\lambda} \right)^{n/2} \frac{e^{i\frac{\pi}{4}\sgn \Lambda}} {|\det \Lambda|^{1/2}} \sum_{0 \leq j \leq N} \frac {\lambda^{-j}} {j!} \left( \frac {\agl[\Lambda^{-1}D, D]} {2i} \right)^j f(0) \nonumber \\
	& \qquad + \mathcal O(\lambda^{-\frac{n}{2}-N-1} \times \sum_{|\alpha| \leq n+2N+3} \sup_{\Rn} |\partial^\alpha a| \times \sum_{|\alpha| \leq n+2N+6} \sup_{\Rn} |\partial^\alpha \varphi|). \label{eq:I1Tem-staphlema}
	\end{align}
	It can be checked that $\sgn \Lambda = \sgn \varphi_{xx}(0)$ and $\det \Lambda = \det \varphi_{xx}(0)$.
	
	We are now almost arrive at \eqref{eq:I-staphlema} except for the explicit computation of the leading term in \eqref{eq:I-staphlema} and \eqref{eq:I1Tem-staphlema}.
	From equality $x = \phi(\Phi(x))$ we know $I = \nabla_y \phi(\Phi(x)) \cdot \nabla_x \Phi(x)$.
	Formula \eqref{eq:xToy-staphlema} implies $\Phi(0) = 0$ and $\nabla_x \Phi(0) = P^T(0)$,
	hence $\det \nabla_y \phi(0) = \det \nabla_y \phi(\Phi(0)) = \big( \det \nabla_x \Phi(0) \big)^{-1} = \big( \det P^T(0) \big)^{-1} = 1$.
	Therefore,
	\begin{equation} \label{eq:fToa-staphlema}
	f(0) = \chi(\phi(0)) \cdot a(\phi(0)) \cdot |\det \nabla_y \phi (0)|
	= \chi(0) \cdot a(0) = a(0).
	\end{equation}
	Combining \eqref{eq:ISplit-staphlema}, \eqref{eq:I2-staphlema}, \eqref{eq:I1Tem-staphlema} and \eqref{eq:fToa-staphlema}, we arrive at the conclusion.
\end{proof}




\end{document}